\documentclass[11pt]{article}
\usepackage{amsmath}
\usepackage{amssymb}
\usepackage{color}
\usepackage{float}
\RequirePackage{algorithmic}
\RequirePackage{algorithm}

\newcounter{remark}[section]

\def\claim{\par\medskip\noindent\refstepcounter{remark}\hbox{\bf Remark \arabic{section}.\arabic{remark}}
	\ 
}
\def\endclaim{
	\par\medskip}

\newtheorem{theorem}{Theorem}

\newtheorem{definition}{Definition}

\def\endpf{{\ \hfill\hbox{\vrule width1.0ex height1.0ex}\parfillskip 0pt
	}}
	\newenvironment{proof}{\noindent{\bf Proof:}}{\endpf}
\begin{document}
\title{Asymptotic comparison of two-stage selection procedures under quasi-Bayesian framework}
\author{Royi Jacobovic\thanks{Department of Statistics; The Hebrew University of Jerusalem; Jerusalem 9190501; Israel.
{\tt royi.jacobovic@mail.huji.ac.il}}}

\date{\today}
		\maketitle
\begin{abstract}
This paper revisits the procedures suggested by Dudewicz and Dalal (1975) and Rinott (1978) which are designed for selecting the population with the highest mean among independent Gaussian populations with unknown and possibly different variances. In a previous paper Jacobovic and Zuk (2017) made a conjecture that the relative asymptotic efficiency of these procedures equals to the ratio of two certain sequences. This work suggests a quasi-Bayesian modelling of the problem under which this conjecture is valid. 
In addition, this paper motivates an open question regarding the extreme value distribution of the maxima of triangular array of independent student-t random variables with an increasing number of degrees of freedom.
\end{abstract}

\section{Introduction}
Consider the problem of a decision maker who has an access to noisy observations taken from a set of populations and has to select the best one. In general, the notion of the best population may be determined with respect to different criteria including highest mean, lowest variance, highest $R$-squared with respect to some target variable and etc. The branch of decision theory which deals with this kind of decisions is known as selection procedures or multiple decision procedures. For books regarding this subject see e.g \cite{Gibbons1999,Gupta1979}. Traditionally, these problems were investigated with small number of populations. Some modern applications of this theory involve gene-expression datasets and discrete event simulation which is a popular methodology for finding the optimal (or near-optimal) system design (e.g. populations).  A comprehensive survey of these applications is provided by \cite{Jacobovic2017}. The point is that both of these modern applications usually involve enormous number of populations.
Respectively, this paper is devoted for analysing the asymptotic performances of two selection procedures which were suggested respectively by Dudewicz and Dalal \cite{Dudewicz1975} and Rinott \cite{Rinott1978}. These procedures were designed in order to find the population with the highest mean among a set of independent Gaussian populations with possibly different variances which are unknown to the user. What we do here is to define their relative asymptotic efficiency as the number of populations tends to infinity. Then, this work suggests a quasi-Bayesian model under which the conjecture of \cite{Jacobovic2017} stating that the asymptotic relative efficiency equals to the limit of a ratio of two sequences $(h_k^1)$ and $(h_k^2)$ to be later explained is valid. The rest is organized as follows: Section \ref{sec: model decription} provides the model description with further details about the above-mentioned procedures. Section \ref{sec: main results} includes the main results and proofs. Finally, Section \ref{sec: discussion} contains a discussion regarding the way that this work motivates an open question about the asymptotic distribution of a sequence of maxima generated by a triangular array of independent student-t random variables (r.v's) with an increasing number of degrees of freedom (d.f's). 
    
\section{Model description} \label{sec: model decription}
Let $\Delta>0$ and consider a sequence of random variables $\Sigma=\{\sigma_i^2\}_{i=1}^\infty$ which are positive with probability one. For each $k\in\mathbb{N}$ the following is a quasi-Bayesian model of samplings drawn from $k+1$ Gaussian populations. We denote these populations by $\Pi^{(k)}_1,\ldots,\Pi^{(k)}_ {k+1}$ and let $X_j^i(k)$ be the $j$th sampling from the population $\Pi^{(k)}_i$. Now, consider a scalar vector $\theta^{(k)}=(\theta_1^{(k)},\ldots,\theta^{(k)}_{k+1})$ which belongs to
\begin{equation*}
\Theta(\Delta):=\left\{
\tilde{\theta}\in\mathbb{R}^{k+1};\min_{1\leq i_1<i_2\leq k+1}|\tilde{\theta}_{i_1}-\tilde{\theta}_{i_2}|>\Delta\right\}\,.
\end{equation*}
The r.v's $X_j^i(k);i=1,\ldots,k+1\ , \ j\in\mathbb{N}$ are such that given $\Sigma$, they are independent and satisfying $X_j^i(k)\sim N(\theta_i^{(k)},\sigma_i^2),\forall i,j\in\mathbb{N}$. 
\subsection{Selection problem}
It is assumed that  $\theta^{(1)},\theta^{(2)},\ldots$ are unknown vectors. More assumptions are that all components of $\Sigma$ are unobservable while $\Delta$ is a parameter which is set by the user with respect to her preferences. In particular, commonly, $\Delta$ is considered as a parameter which reflects the indifference level of the user regarding two populations with distinct means (see e.g. \cite{Bechhofer1954}). With respect to this setup, fix $k\in\mathbb{N}$ and the purpose is to pinpoint the population with the highest mean among the populations $\Pi^{(k)}_1,\ldots,\Pi^{(k)}_ {k+1}$. To this end, a  selection procedure, i.e. a  sampling policy with a selection rule to pinpoint the correct population is required. The requirement from such a rule is to identify the correct population (PCS) with probability which is not less than $p\in(0,1)$ where $p$ is a confidence parameter to be determined by the user. 

\subsection{Two-stage selection procedures}
Two optional two-stage procedures were suggested respectively by Dudewicz and Dalal \cite{Dudewicz1975} (procedure 1) and Rinott \cite{Rinott1978} (procedure 2). Roughly speaking, the first stage is about drawing constant number of samplings $N_0(k)$ from every population. This sample is used to estimate the variance of each population. We denote the estimated variances by $S_1^2(k),\ldots,S_{k+1}^2(k)$. Then, the second stage tells the user to draw more samplings from each population. Importantly, for each population, the additional sample size taken at the second stage is an increasing function of the corresponding empirical variance.  Finally, for each population, the user should average the samplings with respect to some choice of weights and pick the population which is associated with the highest weighted-average. The exact details of these procedures are summarized together at \cite{Jacobovic2017}. For our purposes, we only recall the relevant details. To start with, let $G_{\nu_k}(\cdot)$ and $g_{\nu_k}(\cdot)$ be the c.d.f. and p.d.f of student's-t distribution with $\nu_k:=N_0(k)-1$ degrees of freedom (d.f's). In addition, for each $i$, let $N_i^1(k),N_i^2(k)$ be the number of samplings taken from the $i$th population by the procedures 1 and 2. It is known that for each $l=1,2$
\begin{equation*}\label{eq:sample_size_given_h}
N_i^l(k)=\max\bigg\{N_0(k)+1,\bigg \lceil \bigg(\frac{h_k^l}{\Delta}\bigg)^2 S_i^2(k)\bigg\rceil \bigg\} .
\end{equation*}
where $h_k^{1}$ and $h_k^{2}$ are defined respectively as the unique solutions of the following equations in $h$
\begin{align*}
&p=\int_{-\infty}^\infty G^{k}_{\nu_k}(t+h)g_{\nu_k}(t)dt\\&p=\left[\int_{-\infty}^\infty G_{\nu_k}(t+h)g_{\nu_k}(t)dt\right]^k\,.
\end{align*} 

\subsection{Relative asymptotic efficiency}
To evaluate the performance of the procedures which were introduced earlier, it is possible to consider the expected number of samplings that each of them requires. It is straightforward that when the number of populations is taken to infinity, then the corresponding performance measure also tends to infinity. Therefore, in order to compare the asymptotic performance of these procedures, we are looking at the relative asymptotic efficiency which is defined as the next limit whenever it exists.   

\begin{definition}\label{def: relative efficiency}
	The relative asymptotic efficiency of the above-mentioned procedures is given by the limit 
	\begin{equation*}
	\eta:=\lim_{k\to\infty}\frac{\sum_{i=1}^{k+1}EN_i^2(k)}{\sum_{i=1}^{k+1}EN_i^1(k)}
	\end{equation*}
	if exists. Otherwise, it is not defined. 
\end{definition}

\section{Main results}\label{sec: main results}

The next theorem refers to the case where the sample-size of the first stage is constant with respect to $k$. 
\begin{theorem} \label{thm: constant initial sample size}
	Let $N_0(k)=N_0=\nu+1,\forall k\in\mathbb{N}$. If 
	\begin{enumerate}
		\item $\sigma_i^2\stackrel{d}{=}\sigma_j^2,\forall i,j$
		\item $\sigma^2:=\mathbb{E}\sigma_1^2<\infty$ 
	\end{enumerate}
	then for every $l=1,2$
	\begin{equation*}
	\sum_{i=1}^{k+1}EN_i^2(k)\sim(k+1)\left(\frac{h_k^j\sigma}{\Delta}\right)^2 \ \ \text{as} \ \ k\rightarrow\infty
	\end{equation*} 
	and hence $\eta=\lim_{k\to\infty}\frac{h_k^2}{h_k^1}=2^{2/\nu}$\,. 
\end{theorem}

\begin{proof} 
To start with, observe that $\sigma_1^2,\sigma^2_2,\ldots$ are equally distributed and that
for each $i=1,\ldots,k+1$ the distribution of $S_i^2(k)$ given $\Sigma$ is determined uniquely by $\sigma_i^2$. Especially, since $N_0(k)=N_0,\forall k\in\mathbb{N}$, then this distribution doesn't depend on $k$. Therefore, we shall write that for each $i$, $S^2_i(k)\sim S_1^2$. Now, fix $l\in\{1,2\}$ and notice that $\{h_k^l\}$ is a monotonic sequence such that $h_k^l\to\infty$ as $k\to\infty$. Therefore, we may let $k^l_0:=\min\{k;h_k^l>0\}$ and fix $k>k_0^l$. Then, by their definitions, the r.v's $N_1^l(k),\ldots,N^l_{k+1}(k)$ are equally distributed. This can be used in order to show that
	\begin{align*}
	\alpha_k^l:&=\frac{1}{(k+1)(h_k^l/\Delta)^2}\sum_{i=1}^{k+1}\mathbb{E}N_i^l(k)\\&=\frac{1}{(h_k^l/\Delta)^2}\mathbb{E}N_1^l(k)\\&=\mathbb{E}\max\left\{\frac{N_0+1}{(h_k^l/\Delta)^2},\frac{\left\lceil S_1^2\left(\frac{h_k^l}{\Delta}\right)^2\right\rceil}{(h_k^l/\Delta)^2}\right\}\,.
	\end{align*}
	Now, recall that given $\Sigma$, $S_1^2$ is an unbiased estimator for $\sigma^2_1$. Thus, the law of total expectation implies that 
	\begin{equation*}
	\mathbb{E}S_1^2=\mathbb{E}_\Sigma\mathbb{E}\left(S_1^2|\Sigma\right)=\mathbb{E}\sigma_1^2=\sigma^2<\infty
	\end{equation*}
	which means that $S_1^2$ is an integrable r.v. In addition, since $\{h_k^l\}$ is a non-decreasing sequence which is positive from $k_0^l$, then it can be seen that the maximum inside the expectation is non-negative and bounded by
	
	\begin{equation*}
	Y:=\frac{N_0+1}{(h_{k_0^l}^j/\Delta)^2}+S_1^2+\left(\frac{h_{k^l_0}^j}{\Delta}\right)^{-2}
	\end{equation*}
	which is an integrable r.v. Thus, the dominated convergence theorem (DCT) may be applied in order to derive the limit
	\begin{align*}
	&\exists\lim_{k\rightarrow\infty}\alpha^l_k\\&\stackrel{DCT}{=}\mathbb{E}\left[\lim_{k\rightarrow\infty}\max\left\{\frac{N_0+1}{(h_k^l/\Delta)^2},\frac{\left\lceil S_1^2\left(\frac{h_k^l}{\Delta}\right)^2\right\rceil}{(h_k^l/\Delta)^2}\right\}\right]\\&=\mathbb{E}S_1^2=\sigma^2\,.
	\end{align*} 
	In particular, notice that the derivation of the limit inside the expectation was made by using the facts that $h_k^l\to\infty$ as $k\to\infty$ and $S_1^2>0$ while $N_0$ is a constant. This establishes the first order approximation for $\sum_{i=1}^{k+1}\mathbb{E}N_i^l(k)$ as $k\to\infty$. Therefore, it is an immediate result that 
	\begin{equation*}
	\eta=\lim_{k\to\infty}\frac{h_k^2}{h_k^1}=2^{2/\nu}
	\end{equation*}  
	where the last equality holds due to Theorems 4.1 and 4.2 of \cite{Jacobovic2017}.  
\end{proof}\\
The next theorem refers to the case where the initial sample size tends to infinity as the number of populations tends to infinity.  

\begin{theorem} \label{thm: varying initial sample size}
	Assume that
	\begin{enumerate}
		\item $\sigma_i^2\stackrel{d}{=}\sigma_j^2,\forall i,j$
		\item $\mathbb{E}\left[(\sigma_1^2)^2\right]<\infty$
		\item $2\leq N_0(k)\to\infty$ as $k\to\infty$
		\item $\exists\lim_{k\to\infty}\frac{N_0(k)}{(h_k^l/\Delta)^2}=:L_l<\infty\ \ , \ \ \forall l=1,2$ 
	\end{enumerate}
	and let $Y_l=\max\{L_l,\sigma_1^2\},\forall l=1,2$. Then, for every $l=1,2$
	\begin{equation*}
	\mathbb{E}N_l(k)\sim(k+1)\left(\frac{h_k^l}{\Delta}\right)^2\mathbb{E}Y_l \ \ \text{as} \ \ k\rightarrow\infty
	\end{equation*} 
	and $\eta=\lim_{k\to\infty}\frac{h_k^2}{h_k^1}$ if such limit exists. 
\end{theorem}

\begin{proof}
	Let $l=1,2$. Using the same arguments appeared in the proof of Theorem \ref{thm: constant initial sample size}, for every $k\in\mathbb{N}$ 
	\begin{align*}
	\alpha_k^l:=\frac{\sum_{i=1}^{k+1}\mathbb{E}N_i^l(k)}{(k+1)(h_k^l/\Delta)^2}&=\mathbb{E}\max\left\{\frac{N_0(k)+1}{(h_k^l/\Delta)^2},\frac{\left\lceil S_1^2(k)\left(\frac{h_k^l}{\Delta}\right)^2\right\rceil}{(h_k^l/\Delta)^2}\right\}\,.
	\end{align*} 
	Denote the maximum inside the above-mentioned expectation by $M_k$ and recall that given $\Sigma$, $S_1^2(k)$ is an unbiased estimator of $\sigma_1^2$. Thus, it is known that given $\Sigma$   
	\begin{equation}\label{eq: chi}
	\frac{N_0^l(k)-1}{\sigma_1^2}S_1^2(k)\sim\chi^2_{(N_0^l(k)-1)}\,.
	\end{equation}
	Therefore, Equation \eqref{eq: chi} clearly implies that 
	
	\begin{equation*}
	Var\left[S_1^2(k)\big|\Sigma\right]=\left[\frac{\sigma_1^2}{N_0^l(k)-1}\right]^22\left[N_0^l(k)-1\right]\leq 2\left(\sigma_1^2\right)^2<\infty\,.
	\end{equation*}
	Consequently, since both the first and second moments of $S_1^2(k)$ given $\Sigma$ are finite and constants with respect to $k$, then it may be deduced that $\sup_{k\in\mathbb{N}}\mathbb{E}(M_k^2|\Sigma)$ is bounded by
	\begin{align*}
	\sup_{k\in\mathbb{N}}\mathbb{E}\left\{\left[\frac{N_0^l(k)+1}{(h_k^l/\Delta)^2}+S_1^2(k)+\left(\frac{\Delta}{h_k^l}\right)^2\right]^2\bigg|\Sigma\right\}<\infty
	\end{align*}
	where we have used the facts that $h_k^l\to\infty$ as $k\to\infty$ and $L_l<\infty$ in order to bound the first and third summands inside the  expectation.   
	Thus, by the theorem of de la Vall\'ee Poussin with test function $g(x)=x^2,\forall x\in\mathbb{R}_+$, deduce that given $\Sigma$, $\{M_k;k\in\mathbb{N}\}$ is a sequence of r.v's which are uniformly integrable. Thus, since $N_0(k)\to\infty$ as $k\to\infty$, the strong law of large numbers implies that given $\Sigma$,  $S_1(k)\xrightarrow{1}\sigma^2_1$ as $k\to\infty$. Therefore, Vitali's convergence theorem implies that 
	\begin{equation}\label{eq: conditional limit}
	\lim_{k\to\infty}\mathbb{E}\left(M_k|\Sigma\right)=\mathbb{E}\left(\lim_{k\rightarrow\infty}M_k|\Sigma\right)=\mathbb{E}\left(\max\left\{L_l,\sigma_1^2\right\}\big|\Sigma\right)=\max\left\{L_l,\sigma_1^2\right\}\,.
	\end{equation}
	In addition, observe that 
	
	\begin{equation*}
	\sup_{k\in\mathbb{N}}\mathbb{E}_{\Sigma}\mathbb{E}^2\left(M_k|\Sigma\right)\leq\sup_{k\in\mathbb{N}}\mathbb{E}_{\Sigma}\left[\frac{N_0^l(k)+1}{(h_k^l/\Delta)^2}+\sigma_1^2+\left(\frac{\Delta}{h_k^l}\right)^2\right]^2<\infty
	\end{equation*}
	where we have used the detail that $\sigma_1^2$ is associated with finite second moment along with the fact that for every $k\in\mathbb{N}$, $S_1^2(k)$ is an unbiased estimator of the variance, i.e. $\mathbb{E}\left[S_1^2(k)|\Sigma\right]=\sigma_1^2,\forall k\in\mathbb{N}$. Thus, once again, by the theorem of de la Vall\'ee Poussin with test function $g(x)=x^2,\forall x\in\mathbb{R}_+$, deduce that $\left\{\mathbb{E}(M_k|\Sigma);k\in\mathbb{N}\right\}$ is a sequence of r.v's which are uniformly integrable. Therefore, Vitali's convergence theorem implies that
	
	\begin{align*}
	\lim_{k\rightarrow\infty}\alpha_k^l=\lim_{k\rightarrow\infty}\mathbb{E}_{\Sigma}\mathbb{E}(M_k|\Sigma)=\mathbb{E}_{\Sigma}\lim_{k\rightarrow\infty}\mathbb{E}(M_k|\Sigma)=\mathbb{E}\max\left\{L,\sigma_1^2\right\}
	\end{align*}  
	where the last equality holds due to Equation \eqref{eq: conditional limit}. Finally, if
	\begin{equation*}
	\exists\lim_{k\to\infty}\frac{h_k^2}{h_k^1}=:H \ ,
	\end{equation*} 
	then we shall deduce that $\eta=H$. Especially observe that $EY<\infty$ stems directly from the Theorem assumptions.
\end{proof}

\section{Discussion}\label{sec: discussion}
Theorem \ref{thm: constant initial sample size} establishes the result that for a fixed initial sample size, with the current quasi-Bayesian assumptions, the asymptotic performance of procedure suggested by Dudewicz and Dalal is better than the performance of the procedure suggested by Rinott. While for this case, we have made strict conclusions, for the other case where the initial sample size tends to infinity as the number of populations tends to infinity things are different. Theorem \ref{thm: varying initial sample size} establishes a connection between the asymptotic performance of the procedures to the sequences $(h_k^1)$ and $(h_k^l)$. However, it is still not clear for which sequences $\left(N_0(k)\right)$, condition 4 of Theorem \ref{thm: varying initial sample size} is satisfied? Intuitively speaking, it seems reasonable that when $N_0(k)$ tends to infinity slow enough, then this condition holds, but still this requires further investigation. In addition, for this case, it is not clear, when it exists, what is the limit of $h_k^2/h_k^1$ as $k\to\infty$? To solve this questions, it seems that the same methods used by \cite{Jacobovic2017} in order to derive first order approximations of $(h_k^1)$ and $(h_k^2)$ for the fixed initial sample size should work here as well. However, the main obstacle in that way is the need to specify the extreme value distribution of some triangular arrays of random variables. In the case of $h_k^1$, for each $k$, it is necessary to find the limit distribution of the maximum of $k$ i.i.d student-t r.v's with $\nu_k=N_0(k)-1$ d.f's. Similarly, for the case of $h_k^2$ we should derive the limit distribution of $k$ i.i.d r.v's such that each of them is distributed like a sum of two independent student-t r.v's with $\nu_k$ d.f's. To the best of our knowledge non of these triangular arrays have references in literature.  For existing literature about extreme value distributions of triangular arrays see e.g. \cite{Anderson1997,Bose2008,Freitas2003,Hashrova2005,Hsing1996}.

\end{document}